\documentclass[a4paper]{amsart}
\usepackage{amsmath, amsthm, amssymb}
\usepackage[T1]{fontenc}
\usepackage[utf8]{inputenc}     
\usepackage[dvipsnames]{xcolor}
\usepackage{xstring}            
\usepackage{enumerate}
\usepackage[
  colorlinks=true, linkcolor=blue, linkbordercolor=blue, citecolor=red, citebordercolor=red, urlcolor=blue, linktocpage=true
]{hyperref}
\usepackage[capitalise, nameinlink, noabbrev, nosort]{cleveref} 
\usepackage[lite]{amsrefs}  

\usepackage{mathtools}          
\usepackage{bm}           
\usepackage[shortcuts]{extdash}       
\usepackage{subcaption}         
\usepackage{accents}
\usepackage{xspace}

\usepackage{xparse} 

\usepackage{tikz-cd}
\usetikzlibrary{decorations.markings}
\tikzset{double line with arrow/.style args={#1,#2}{decorate,decoration={markings,%
mark=at position 0 with {\coordinate (ta-base-1) at (0,-2pt);
\coordinate (ta-base-2) at (0,2pt);},
mark=at position 1 with {\draw[#1] (ta-base-1) -- (0,-2pt);
\draw[#2] (ta-base-2) -- (0,2pt);
}}}}

\newcommand{\gax}{\Gamma(G,X)}

\newcommand{\myllangle}{\langle\!\langle}
\newcommand{\myrrangle}{\rangle\!\rangle}

\DeclarePairedDelimiter\abs{\lvert}{\rvert}        
\DeclarePairedDelimiter\norm{\lVert}{\rVert}        
\DeclarePairedDelimiter\angles{\langle}{\rangle}    
\DeclarePairedDelimiter\aangles{\myllangle}{\myrrangle}
\DeclarePairedDelimiter\braces{\{}{\}}            

    \newcommand{\bigbraces}[1]{\braces[\big]{#1}}
    
\newcommand{\bigmid}{\mathrel{\big|}}            

\DeclareMathOperator{\id}{id}                

\DeclareMathOperator{\aut}{Aut}                
\DeclareMathOperator{\out}{Out}                

\DeclareMathOperator{\Hom}{Hom}
\DeclareMathOperator{\Out}{Out}

\newcommand{\cent}{C}

\DeclareMathOperator{\Gl}{GL}            

\theoremstyle{plain}
\newtheorem{thm}{Theorem}[section]        
\newtheorem{prop}[thm]{Proposition}        \newtheorem{proposition}[thm]{Proposition}
\newtheorem{lem}[thm]{Lemma}            \newtheorem{lemma}[thm]{Lemma}
        \newtheorem{corollary}[thm]{Corollary}

        \newtheorem{problem}[thm]{Problem}

\newtheorem*{thm*}{Theorem}            \newtheorem*{theorem*}{Theorem}
\newtheorem*{prop*}{Proposition}        \newtheorem*{proposition*}{Proposition}
\newtheorem*{lem*}{Lemma}            \newtheorem*{lemma*}{Lemma}
\newtheorem*{cor*}{Corollary}            \newtheorem*{corollary*}{Corollary}
\newtheorem*{qu*}{Question}            \newtheorem*{question*}{Question}
\newtheorem*{conj*}{Conjecture}            \newtheorem*{conjecture*}{Question}
\newtheorem*{prob*}{Problem}        \newtheorem*{problem*}{Problem}
\newtheorem*{fact*}{Fact}
\newtheorem*{claim*}{Claim}
\newtheorem*{case*}{Case}

\numberwithin{equation}{section}

\newtheorem{alphthm}{Theorem}            

\newtheorem{alphcor}[alphthm]{Corollary}           

\newtheorem{alphproblem}[alphthm]{Problem}

\theoremstyle{definition}
        \newtheorem{definition}[thm]{Definition}

\newtheorem*{de*}{Definition}            \newtheorem{definition*}{Definition}
\newtheorem*{notation*}{Notation}
\newtheorem*{conv*}{Convention}            \newtheorem*{convention*}{Convention}

\theoremstyle{remark}
            \newtheorem{remark}[thm]{Remark}
            \newtheorem{example}[thm]{Example}

\newtheorem*{rmk*}{Remark}        \newtheorem*{remark*}{Remark}

\DeclareMathAlphabet{\mathbit}{OT1}{cmr}{bx}{it}

\crefname{thm}{Theorem}{Theorems}              \crefname{theorem}{Theorem}{Theorems}
\crefname{prop}{Proposition}{Propositions}     \crefname{proposition}{Proposition}{Propositions}
\crefname{lem}{Lemma}{Lemmas}                  \crefname{lemma}{Lemma}{Lemmas}
\crefname{rmk}{Remark}{Remarks}                \crefname{remark}{Remark}{Remarks}
\crefname{cor}{Corollary}{Corollaries}         \crefname{corollary}{Corollary}{Corollaries}
\crefname{qu}{Question}{Questions}             \crefname{question}{Question}{Questions}
\crefname{conj}{Conjecture}{Conjectures}       \crefname{conjecture}{Conjecture}{Conjectures}
\crefname{prob}{Problem}{Problems}             \crefname{problem}{Problem}{Problems}
\crefname{fact}{Fact}{Facts}
\crefname{claim}{Claim}{Claims}
\crefname{case}{Case}{Cases}
\crefname{alphthm}{Theorem}{Theorems}          \crefname{alphcor}{Corollary}{Corollaries}
\crefname{alphprop}{Proposition}{Propositions}
\crefname{alphprolem}{Problem}{Problems}

\newcommand{\Cat}[1]{\ifmmode \text{\normalfont \textbf{#1}} \else {\normalfont \textbf{#1}}\fi}
\newcommand{\mhyphen}{\textnormal{-}}
\newcommand{\variable}{\,\mhyphen\,}

\newcommand\MakeComments[3]{
  \newcounter{#2comment}\addtocounter{#2comment}{1}%
  \newcommand{#1}[1]{%
     \textbf{\color{#3}(\StrChar{#2}{1}\arabic{#2comment})}%
     \marginpar{\tiny\raggedright\textbf{\color{#3}(\StrChar{#2}{1}\arabic{#2comment}) #2:} ##1}%
    \addtocounter{#2comment}{1}%
  }%
}

\definecolor{newpurple}{rgb}{0.8, 0, 0.9}

\ExplSyntaxOn
\cs_new_protected:Nn \egreg_embolden_command:N
 {\cs_set_eq:cN { __egreg_ \cs_to_str:N #1 : } #1
  \cs_set_protected:Npn #1 { \bm { \use:c { __egreg_ \cs_to_str:N #1 : } } }}
\cs_new_protected:Nn \egreg_embolden_char:n
 {\exp_args:Nc \mathchardef { __egreg_#1: } = \mathcode`#1 \scan_stop:
  \cs_set_protected:cn { __egreg_#1_bold: } { \bm { \use:c { __egreg_#1: } } }
  \char_set_active_eq:nc { `#1 } { __egreg_#1_bold: }
  \mathcode`#1 = "8000 \scan_stop:}
\cs_new_protected:Nn \egreg_embolden:
 {\clist_map_function:nN
   {
    A,B,C,D,E,F,G,H,I,J,K,L,M,N,O,P,Q,R,S,T,U,V,W,X,Y,Z,
    a,b,c,d,e,f,g,h,i,j,k,l,m,n,o,p,q,r,s,t,u,v,w,x,y,z,
    *,<,>,/,-,1,2,3,4,5,6,7,8,9,0
   }
   \egreg_embolden_char:n
  \clist_map_function:nN
   {
    \alpha,\beta,\gamma,\delta,\epsilon,\varepsilon,\zeta,\eta,\theta,\vartheta,\iota,\kappa,\lambda,\mu,\nu,\xi,\pi,\varpi,\rho,\varrho,\sigma,\varsigma,\tau,\upsilon,\phi,\varphi,\chi,\psi,\omega,
    \Gamma,\varGamma,\Delta,\varDelta,\Theta,\varTheta,\Lambda,\varLambda,\Xi,\varXi,\Pi,\varPi,\Sigma,\varSigma,\Upsilon,\varUpsilon,\Phi,\varPhi,\Psi,\varPsi,\Omega,\varOmega,
    \cup,\cap,
    \neq,\cong,\ncong,
    \in,
    \subset,\subseteq,\supset,\supseteq,\subsetneq,\supsetneq,\nsubseteq,\nsupseteq,
    \leq,\geq,\lneq,\gneq,\lhd,\rhd,\trianglelefteq,\trianglerighteq,\triangleright,\trianglerighteq,\circ,
   }
   \egreg_embolden_command:N}
\NewDocumentCommand{\crse}{m}
 {\group_begin:
  \egreg_embolden:
  #1
  \group_end:}
\ExplSyntaxOff

\newcommand{\cid}{\mathbf{id}}

\DeclareMathOperator{\cAut}{Aut_{\mathbf{Crs}}}

\newcommand{\cmp}{\circ}
\newcommand*{\op}[1]{#1^{\scriptscriptstyle\rm T}}

\makeatletter
\def\varcrs{\@ifnextchar[{\@withvarcrs}{\@withoutvarcrs}}
\def\@withvarcrs[#1]#2{\mathcal{E}_{\rm #2}^{\rm #1}}
\def\@withoutvarcrs#1{\mathcal{E}_{\rm #1}}
\def\varFcrs{\@ifnextchar[{\@withvarFcrs}{\@withoutvarFcrs}}
\def\@withvarFcrs[#1]#2{\mathcal{F}_{\rm #2}^{\rm #1}}
\def\@withoutvarFcrs#1{\mathcal{F}_{\rm #1}}
\makeatother

\MakeComments{\asnote}{Ashot}{Purple}
\MakeComments{\alenote}{Alex}{Orange}
\MakeComments{\fnote}{Fede}{ForestGreen}

\DeclareMathOperator{\inn}{Inn}
\DeclareMathOperator{\lox}{\mathcal{L}}
\DeclareMathOperator{\qout}{QOut}
\DeclareMathOperator{\qi}{QI}

\newcommand{\h}{\hookrightarrow_h}

\newcommand{\dbw}{d_{\rm bw}}

\begin{document}

\title[On automorphisms, quasimorphisms and coarse automorphisms]{On automorphisms, quasimorphisms, and coarse automorphisms of acylindrically hyperbolic groups}


\begin{abstract}
    We investigate the action of the automorphism group of an acylindrically hyperbolic group $G$ on its space of homogeneous quasimorphisms, and identify its kernel with the subgroup of ``strongly commensurating'' automorphisms.
    We deduce that if $G$ has no non-trivial finite normal subgroups then it has sufficiently many quasimorphisms to recognize whether an automorphism is inner. As consequences, we show that $\Out(G)$ acts faithfully on the kernel of the comparison map in bounded cohomology and it embeds in (several) groups of coarse automorphisms.
\end{abstract}


\author{Ashot Minasyan}
\address[A. Minasyan]{CGTA, School of Mathematical Sciences, University of Southampton, Highfield, Southampton, SO17~1BJ, United Kingdom}
\email{aminasyan@gmail.com}

\author{Alessandro Sisto}
 \address[A. Sisto]{Department of Mathematics, Heriot-Watt University and Maxwell Institute for Mathematical Sciences, Edinburgh, UK}
    \email{a.sisto@hw.ac.uk}
    
\author{Federico Vigolo}
\address[F. Vigolo]{Mathematisches Institut, Georg-August-Universität Göttingen, Bunsenstr. 3-5, 37073 Göttingen, Germany.}
\email{federico.vigolo@uni-goettingen.de}


\maketitle

\section{Introduction}
The goal of this note is to study the natural action of the group of outer automorphisms $\out(G)$, of an acylindrically hyperbolic group $G$, on the space of quasimorphisms of $G$, and to leverage this to find natural embeddings of $\out(G)$ into certain groups of ``coarse automorphisms''.

Recall that a \emph{quasimorphism} of $G$ is a map $q\colon G\to\mathbb R$ such that the \emph{defect}
\[
D(q)\coloneqq \sup_{g,h\in G}\abs{q(gh)-q(g)-q(h)}
\]
is finite.
A quasimorphisms $q$ is \emph{homogeneous} if $q(g^k)=kq(g)$, for all $g\in G$ and $k\in \mathbb Z$. We denote
\[
Q(G)\coloneqq\{q\colon G\to \mathbb R \mid q \text{ is a quasimorphism}\} ~\text{ and }
\]
\[
Q_h(G)\coloneqq\{q\in Q(G) \mid q \text{ is homogeneous}\}.
\]

Two functions $f,f'\colon X\to(Y,d)$ taking values in a metric space are \emph{close} if the distance $d(f(x),f'(x))$ remains uniformly bounded as $x\in X$ varies.
It is well known that every quasimorphism $q$ is close to a unique homogeneous quasimorphism $\tilde q$ (see \cref{lem:qm_close_to_hqm}).
In particular, the study of $Q_h(G)$ is equivalent to the study of the space of quasimorphisms considered up to closeness.

There are several reasons to be interested in the set of quasimorphisms of a given group $G$. For instance, they are important in the study of stable commutator length, bounded cohomology, and in showing the existence of unbounded bi-invariant metrics on $G$ \cites{calegari2009scl,frigerio2017bounded,coarse_groups}.
Of course, $\aut(G)$ acts on $Q_h(G)$ by precomposition, and it is natural to study the properties of this action.
Interesting results in this direction are proven in \cite{fournier2023Aut}, where many examples of groups are given where $\aut(G)\curvearrowright Q_h(G)$ has global fixed points (see \cites{brandenbursky2019aut,karlhofer2021autinvariant} for earlier results).

The main objective of this note is orthogonal to that of \cite{fournier2023Aut}. Namely, we investigate the following:

\begin{alphproblem}\label{prob:intro}
    For a given group $G$, what is the kernel of the action of $\aut(G)$ on $Q_h(G)$?
\end{alphproblem}

Our main contribution is a complete solution to \cref{prob:intro} in the class of acylindrically hyperbolic groups. Namely, we say that an automorphism $\phi:G \to G$ is \emph{strongly commensurating} if for every $g \in G$ there exist $m \in \mathbb{N}$ and $h \in G$ such that $\phi(g^m)=h g^m h^{-1}$ in $G$. We will denote by $\aut_{sc}(G)$ the set of all strongly commensurating automorphisms of $G$.
It is not hard to see that $\aut_{sc}(G)$ is a normal subgroup of $\aut(G)$ that is always contained in the kernel of the action of $\aut(G)$ on $Q_h(G)$ (see \cref{lem:sc_auts_form_a_normal_sbgps} below).
In this note we prove that for acylindrically hyperbolic groups the other containment also holds.

\begin{alphthm}\label{thm:intro:main_thm}
    Let $G$ be an acylindrically hyperbolic group, and let  $\phi\in\aut(G)$. Then the following are equivalent:
    \begin{itemize}
        \item[(i)] $\phi\notin\aut_{sc}(G)$,
        \item[(ii)] there exists $q\in Q_h(G)$ such that $q\neq q\circ \phi$,
        \item[(iii)] there exist $q\in Q_h(G)$ and $g\in G$ such that $q|_{\langle g \rangle}$ is unbounded, $(q \circ \phi)|_{\langle g \rangle}$ is zero and the difference $(q- q\circ \phi)\colon G\to \mathbb R$ is not a homomorphism.
    \end{itemize}
    In particular, $\aut_{sc}(G)$ is the kernel of the action $\aut(G)\curvearrowright Q_h(G)$.
\end{alphthm}

The class of acylindrically hyperbolic groups is very large and contains a number of groups of special interest, see \cites{osin_acylindrically_2016,Osin-survey}. For instance, non-elementary (relatively) hyperbolic groups, most mapping class groups, and $\out(F_n)$, for $n\geq 2$, are all acylindrically hyperbolic.

In condition (iii) of \cref{thm:intro:main_thm}, the fact that $(q- q\circ \phi)$ is not a homomorphism is particularly interesting in view of the connection between quasimorphisms and bounded cohomology. Namely, there is an exact sequence
\[
    \{0\}\longrightarrow
    \Hom(G,\mathbb R)\longrightarrow
    Q_h(G)\longrightarrow
    H^2_{b}(G;\mathbb R)\longrightarrow
    H^2(G;\mathbb R),
\]
where the second map is the natural inclusion and the last one is the \emph{comparison map} (\emph{i.e.}\ the homomorphism induced by embedding the complex of bounded cochains into that of all cochains), see \emph{e.g.}\ \cite{calegari2009scl}*{Theorem 2.50} or \cite{frigerio2017bounded}*{Corollary 2.11}. This sequence is one of the most powerful tools at disposal in the study of $H^2_{b}(G;\mathbb R)$, because it reduces it to understanding the ordinary cohomology group $H^2(G;\mathbb R)$ and the kernel of the comparison map, which is isomorphic to $Q_h(G)/\Hom(G,\mathbb R)$. The following immediate consequence of \cref{thm:intro:main_thm} is therefore of interest.

\begin{alphcor}
    If $G$ is an acylindrically hyperbolic group, then $\aut_{sc}(G)$ is the kernel of the natural action of $\aut(G)$ on $Q_h(G)/\Hom(G,\mathbb R)$.
\end{alphcor}

For every group $G$, the group of inner automorphisms $\inn(G) \subseteq \aut_{sc}(G)$ is contained in the kernel of the action of $\aut(G)$ on $Q_h(G)$. It follows that this action descends to a natural action of the outer automorphism group $\out(G)$ on $Q_h(G)$.
If $G$ is acylindrically hyperbolic, \cref{thm:intro:main_thm} shows that the kernel of this action of $\out(G)$ is $\aut_{sc}(G)/\inn(G)$. Moreover,  this kernel is necessarily finite when $G$ is finitely generated, see  \cite{antolin2016commensurating}*{Remark~ 7.6}.

\begin{remark}
    If $G$ contains non-trivial finite normal subgroups the inclusion of  $\inn(G)$ in $\aut_{sc}(G)$ may be strict (\emph{e.g.},\ $F_2 \times \mathbb{Z}/2\mathbb{Z}$ has strongly commensurating automorphisms that are not inner).
\end{remark}

It is known that any acylindrically hyperbolic group $G$ contains a unique maximal finite normal subgroup $K(G)$ \cite{dahmani2017hyperbolically}*{Theorem 2.24}, and moreover when $K(G)$ is trivial we have $\aut_{sc}(G)=\inn(G)$ \cite{antolin2016commensurating}*{Theorem 7.5}.
We thus obtain the following.

\begin{alphcor}\label{cor:faithful and fixed points}
    Let $G$ be acylindrically hyperbolic with no non-trivial finite normal subgroups. Then the canonical action of $\Out(G)$ on $Q_h(G)/\Hom(G,\mathbb R)$ is faithful. A fortiori, the action of $\out(G)$ on $Q_h(G)$ is faithful as well.
\end{alphcor}

\smallskip

Following Hartnick--Schweitzer \cite{HS}, there is an interesting way to reformulate the equivalence of (i) and (ii) in \cref{thm:intro:main_thm} in terms of the mapping from $\aut(G)$ to the group of ``quasioutomorphisms'' of $G$.
Specifically, we say that a function $f\colon G\to H$ is an \emph{HS-quasimorphism} if for every quasimorphism $q\in Q(H)$ the composition $q\circ f$ is also a quasimorphism. Two HS-quasimorphisms $f,f'\colon G\to H$ are \emph{equivalent} if the modulus $\abs{q\circ f - q\circ f'}$ is bounded for every $q\in Q(H)$. It is easy to verify that compositions of HS-quasimorphism are HS-quasimorphisms, and that composition preserves the equivalence of HS-quasimorphisms. It follows that the set equivalence classes of invertible HS-quasimorphisms from $G$ to itself is a group under composition, where $f\colon G\to G$ is \emph{invertible} if there is an HS-quasimorphism $g\colon G\to G$ such that both $f\circ g$ and $g\circ f$ are equivalent to the identity map $\id_G$.
This group is called the \emph{quasioutomorphism group} of $G$, denoted
\[
\qout(G)\coloneqq\left\{f\colon G\to G\ \middle|\ \text{$f$ is an invertible HS-quasimorphism}\right\}/\text{equivalence}.
\]
Of course, inclusion induces a natural homomorphism $\aut(G)\to \qout(G)$. 

By definition, an automorphism $\phi$ is in the kernel of $\aut(G)\to \qout(G)$ if and only if $\abs{q - q\circ\phi}$ is bounded for every $q\in Q(G)$. Passing to the homogeneous quasimorphism close to $q$, this is equivalent to saying that $\phi$ fixes every homogeneous quasimorphism $q\in Q_h(G)$ (if $q(g)\neq q\circ \phi(g)$, then $\abs{q(g^n)-(q\circ\phi)(g^n)}$ grows to infinity with $n$). In other words, the equivalence of (i) and (ii) in \cref{thm:intro:main_thm} and the part of \cref{cor:faithful and fixed points} concerning the action on $Q_h(G)$ can be restated as follows.

\begin{alphcor}\label{cor:intro:qout}
    Let $G$ be an acylindrically hyperbolic group then $\aut_{sc}(G)$ is the kernel of the natural homomorphism $\aut(G)\to\qout(G)$.
    Moreover, if $K(G)=\{1\}$, this homomorphism descends to a natural embedding $\out(G)\hookrightarrow\qout(G)$.
\end{alphcor}

\begin{remark}
    For the free group $F_n$, \cref{cor:intro:qout} was proved in \cite{HS}*{Proposition 5.1}.
\end{remark}

\cref{cor:faithful and fixed points} also has an interesting consequence regarding \emph{bi-invariant metrics} on groups (\emph{i.e.}\ metrics that are invariant under both left and right multiplication). If $G$ is normally generated by a finite set $S$ (that is,  $G=\aangles{S}$), then the word metric associated with the (usually infinite) generating set $\overline{S}\coloneqq \bigcup_{g\in G} gSg^{-1}$ is bi-invariant. Choosing another finite normally generating set is easily seen to yield a bi-Lipschitz equivalent metric. We denote one such bi-invariant word metric by $\dbw$, and observe that $\aut(G)$ acts by Lipschitz transformations on $(G,\dbw)$. In particular, we have a natural homomorphism
$\aut(G)\to \qi(G,\dbw)$, where
\[\qi(G,\dbw)\coloneqq \{f\colon (G,\dbw)\to(G,\dbw)\mid f \text{ is a quasi-isometry}\}/\text{closeness}.
\]
Once again, this homomorphisms quotients through $\out(G)$, because for any given $g \in G$, using the bi-invariance of $\dbw$, we have
\begin{equation}\label{eq:conjugation is close}
\dbw(h,ghg^{-1})\leq \dbw(h,gh) + \dbw(gh,ghg^{-1})= 2\,\dbw(1,g)
\end{equation}
is bounded by a constant depending only on $g$, which shows that the conjugation by $g$ is close to the identity map.

Observe that if $q\in Q_h(G)$ is any homogeneous quasimorphism and $g\in G$ is an element with $\abs{g}_{\rm bw}=k$, as witnessed by a decomposition $g = s_1^{h_1}\cdots s_k^{h_k}$, for some $s_i\in S^{\pm 1}$ and $h_i \in G$, then
\[
\abs{q(g)} = \abs{q(s_1^{h_1}\cdots s_k^{h_k})} \leq \sum_{i=1}^k\abs{q(s_i^{h_i})} + (k-1)D(q)\leq k \left(D(q)+\max_{s\in S} \abs{q(s)}\right).
\]
Therefore, for any $g_1,g_2 \in G$, with $\dbw(g_1,g_2) \ge 1$, and $C \coloneqq D(q)+\max_{s\in S} \abs{q(s)}$, we have
\[
|q(g_1)-q(g_2)| \le |q(g_1^{-1}g_2)|+D(q) \le (C+1)\dbw(g_1,g_2).
\]
Since $\dbw(g_1,g_2) \in \mathbb{N} \cup \{0\}$, this shows that $q$ is Lipschitz as a function $q\colon (G,\dbw)\to (\mathbb R,\abs{\variable})$. \cref{cor:faithful and fixed points} then has the following consequence.

\begin{alphcor}\label{cor:intro:qi}
    Let $G$ be a finitely normally generated  acylindrically hyperbolic group with $K(G)=\{1\}$. Then the homomorphism $\out(G)\to \qi(G,\dbw)$ is injective.
\end{alphcor}

\begin{remark}
    The metric $\dbw$ is a very natural metric to consider in the study of bi-invariant metrics, because it is `coarsely maximal' in the sense that if $d$ is any bi-invariant metric on $G$ then there is some constant $C >0$ such that $d\leq C\dbw$ (this is seen by letting $C\coloneqq \max_{s\in S}d(1,s)$ and applying bi-invariance and the triangle inequality on $d$).
    In particular, if $\dbw$ is bounded on $G$, then $G$ does not admit \emph{any} unbounded bi-invariant metric.

    It is in general very hard to understand whether $\dbw$ is unbounded for a given group $G$, and one of the most flexible methods for proving it is by showing that $Q_h(G)\neq \{0\}$ (see \cite{coarse_groups}*{Section 8.3} and references therein for an entry point to this subject). \cref{cor:intro:qi} can be seen as saying that for  acylindrically hyperbolic groups $\dbw$ is not only unbounded, but also fine enough to recognise innerness.
\end{remark}

There is a further point of view that unifies both \cref{cor:intro:qout} and \cref{cor:intro:qi}, namely that of coarse automorphisms groups. Postponing the definitions to \cref{sec:crse aut}, we record here the following.

\begin{alphcor}\label{cor:intro:cAut}
    Let $G$ be an acylindrically hyperbolic group with $K(G)=\{1\}$. If $\mathcal E$ is a connected, $\aut(G)$-invariant and equi bi-invariant coarse structure on $G$, with $\mathcal E\subseteq \mathcal E_{G\to\mathbb R}$, then the natural homomorphism $\out(G)\to \cAut(G,\mathcal E)$ is injective.
\end{alphcor}

This solves \cite{coarse_groups}*{Problem 12.1.4} for  acylindrically hyperbolic groups without non-trivial finite normal subgroups.

\begin{remark}
    A possibly misleading but certainly very inspiring observation is that \cref{cor:intro:cAut} is already known to be true for $\mathbb Z^n$ with its natural coarse structure. In fact, in this case $\out(\mathbb Z^n)\cong\Gl(n,\mathbb Z)$, $\cAut(\mathbb Z^n,\varcrs{\norm\variable})\cong \Gl(n,\mathbb R)$, and the injection $\out(\mathbb Z^n)\hookrightarrow \cAut(\mathbb Z^n,\varcrs{\norm\variable})$ is the natural inclusion of $\Gl(n,\mathbb Z)$ into $\Gl(n,\mathbb R)$ (\cite{coarse_groups}*{Corollary 1.3.7}).

    The ``misleading'' part of the above observation is that for non-abelian $G$ the groups of coarse automorphisms seem to be very mysterious objects, much unlike $\Gl(n,\mathbb R)$.
\end{remark}

\begin{remark}
    One curious consequence of \cite{chifan2024small} (see also \cite{ollivier2007kazhdan}) and \cref{cor:intro:cAut} is that every countable group appears as a subgroup of  $\cAut(G,\mathcal E)$ for some finitely generated acylindrically hyperbolic $G$ with Kazhdan's property (T).
\end{remark}

\subsection*{Idea of the proof and structure of the paper}
We now discuss the proof of \cref{thm:intro:main_thm} in the case when $G$ contains no non-trivial finite normal subgroups (so that  $\aut_{sc}(G)=\inn(G)$). This case is considered in \cref{thm:main_with_trivial_radical}, which we will show later to imply \cref{thm:intro:main_thm}.  The starting point is a theorem of the first two authors and Antol\'{i}n, which shows that $\phi \in \aut(G)$ is inner if and only if $g$ and $\phi(g)$ are commensurable for all loxodromic $g\in G$ \cite{antolin2016commensurating}.
Using  a variety of tools from the theory of acylindrically hyperbolic groups, including the Algebraic Dehn Filling Theorem of Dahmani--Guirardel--Osin \cite{dahmani2017hyperbolically}, we improve this statement by showing that for each $\phi \in \aut(G)\smallsetminus \inn(G)$ there is a ``special'' loxodromic element $g \in G$ and such that $\phi(g)$ is also special loxodromic for a given acylindrical action on a hyperbolic space (this results may be of independent interest, see \cref{prop:special non commensurable}).
Once this is done, the equivalence of (i) and (ii) in \cref{thm:intro:main_thm} is shown directly by using the quasimorphism extension theorem of Hull--Osin \cite{hull2013induced}. Proving that these conditions imply (iii) as well requires one extra trick to make sure that the difference $q-q\circ\phi$ cannot be a homomorphism.

This outline of proof is reflected in the structure of this note: in \cref{sec:prelims} we recover all the relevant results about acylindrically hyperbolic groups and black-box them in self-contained statements. In \cref{sec:proof} we use them to prove \cref{thm:intro:main_thm}.

In \cref{sec:crse aut}, we recall the notation from \cite{coarse_groups} to make sense of \cref{cor:intro:cAut} and show how it follows from the main theorem. In the final section we discuss some questions that arise quite naturally in this setup and we include a result telling apart two coarsifications of the rank $n$ free group $F_n$ (\cref{thm:strict containment}). This last fact provides an answer to a special case of a question in \cite{coarse_groups}*{Remark 12.3.13}.

\subsection*{Acknowledgments}
The authors would like to thank Francesco Fournier-Facio for his helpful comments and for pointing out that commutator length does not bound its stable counterpart (\cref{thm:strict containment}), and Jarek K\k{e}dra for providing the last ingredient needed for the proof of \cref{thm:strict containment} (\cref{lem:commutator vs cancellation length}).

The authors also thank the Isaac Newton Institute for Mathematical Sciences, Cambridge, for support and hospitality during the programme ``Operators, Graphs, Groups'' where work on this paper was undertaken. This work was supported by EPSRC grant no EP/Z000580/1.


\section{Preliminaries}\label{sec:prelims}
\subsection{Quasimorphisms}
Let $G$ be a group. 
Recall that the quasimorphisms $q,q':G \to \mathbb{R}$ are said to be \emph{close} if 
$\sup_{g \in G}\abs{q'(g)-q(g)}<\infty$. It is well-known that any quasimorphism is close to a homogeneous one.

\begin{lemma}\label{lem:qm_close_to_hqm} 
Every quasimorphism on a group $G$ is close to a unique homogeneous quasimorphism. More precisely, if $q \in Q(G)$ then the function $\overline{q}:G \to \mathbb{R}$, defined by
\[\overline{q}(g)=\lim_{n \to \infty}\frac{q(g^n)}{n},~\text{ for all } g \in G,\]
is a homogeneous quasimorphism close to $q$. 
Moreover, if $r,r' \in Q_h(G)$ are close to each other then $r=r'$.
\end{lemma}

\begin{proof} The fact that $\overline{q}$ is a homogeneous quasimorphism close to $q$ is proved in \cite{calegari2009scl}*{Lemma 2.21} or \cite{frigerio2017bounded}*{Proposition 2.10}.

Now suppose that for $r,r' \in Q_h(G)$ there is $C \ge 0$ such that $\abs{r(g)-r'(g)} \le C$ for all $g \in G$. Then for every $h \in G$ and all $n \in \mathbb{N}$ we have
\[|r(h)-r'(h)|=\frac1n |r(h^n)-r'(h^n)| \le \frac{C}n,\]
hence $r(h)=r'(h)$, as claimed.    
\end{proof}

Recall, from the introduction, that $\aut_{sc}(G)$ denotes the set of all strongly commensurating automorphisms of $G$.
\begin{lemma}\label{lem:sc_auts_form_a_normal_sbgps}
    If $G$ is  any group then
\begin{itemize}
    \item[(i)] $\aut_{sc}(G) $ is a normal subgroup of $\aut(G)$, containing $\inn(G)$;
    \item[(ii)] $\aut_{sc}(G) $ acts trivially on $Q_h(G)$, that is, for all $q \in Q_h(G)$ and all $\phi \in \aut_{sc}(G)$ we have $q \circ \phi=q$.
\end{itemize}
\end{lemma}

\begin{proof} Verification of claim (i) is straightforward and left to the reader.

For claim (ii), suppose that $q \in Q_h(G)$ has defect $D \ge 0$. One easily checks that for all $x,y \in G$ we have 
$\abs{q(xy)-q(yx)}\le 2D$, which implies that
\begin{equation}\label{eq:quasimorphi_almost_commut}
    \abs{q(g)-q(hgh^{-1})}\le 2D,~\text{ for all }g,h\in G.
\end{equation}

Consider any $\phi \in \aut_{sc}(G)$ and any element $g \in G$. Then there exist $h \in G$ and $m \in \mathbb{N}$ such that $\phi(g^m)=hg^mh^{-1}$. In view of \eqref{eq:quasimorphi_almost_commut}, for each $n \in \mathbb{N}$ we have
\[
\abs{q(g) - q(\phi(g))}
=\frac {1}{mn} \abs{q(g^{mn}) - q(\phi(g)^{mn})}=
\frac {1}{mn} \abs{q(g^{mn}) - q(hg^{mn}h^{-1})} \leq \frac{2D}{mn}.
\]
Since the right-hand side can be made arbitrarily small by choosing large enough $n$, we can conclude that
$q(g)=q(\phi(g))$, for all $g \in G$, as required.   
\end{proof}

\subsection{Acylindrically hyperbolic groups}
We now briefly recall the definition and basic results about acylindrically hyperbolic groups.

Following Osin \cite{osin_acylindrically_2016}, we say that a group $G$ is \emph{acylindrically hyperbolic} if there is a (possibly infinite) generating set $X$ of $G$ such that the Cayley graph $\Gamma(G,X)$ is a hyperbolic metric space, the natural action of $G$ on $\Gamma(G,X)$ is \emph{acylindrical} and the visual boundary $\partial \Gamma(G,X)$ contains more than two points (equivalently, $G$ admits a non-elementary acylindrical action on a hyperbolic metric space, see \cite{osin_acylindrically_2016}*{Theorem~1.2}).

An equivalent way to study acylindrically hyperbolic groups proceeds via {hyperbolically embedded subgroups}, introduced by Dahmani, Guirardel and Osin in \cite{dahmani2017hyperbolically}. Let $\{H_\lambda\}_{\lambda \in \Lambda}$ be a family of subgroups of a group $G$, and let $X \subseteq G$ be a \emph{relative generating set} of $G$ with respect to this family (\emph{i.e.}\ $G$ is generated by the subset $X \bigcup_{\lambda \in \Lambda} H_\lambda$). The family of subgroups $\{H_\lambda\}_{\lambda \in \Lambda}$ is \emph{hyperbolically embedded in $G$ with respect to  $X$}, written  $\{H_\lambda\}_{\lambda \in \Lambda} \h (G,X)$, if the Cayley graph $\Gamma(G, X \sqcup \bigsqcup_{\lambda \in \Lambda} H_\lambda)$ is hyperbolic and for each $\lambda \in \Lambda$ the metric space $(H_\lambda,\widehat{\mathrm{d}}_\lambda)$ is locally finite, where
$\widehat{\mathrm{d}}_\lambda$ denotes the \emph{relative metric}  on $H_\lambda$ (the precise definition for which can be found  in  \cite{osin_acylindrically_2016}*{Subsection~2.4}). We will also write $\{H_\lambda\}_{\lambda \in \Lambda} \h G$ if there is a relative generating set $X$ of $G$ such that $\{H_\lambda\}_{\lambda \in \Lambda} \h (G,X)$.

By \cite{osin_acylindrically_2016}*{Theorem~1.2}, a group is acylindrically hyperbolic if and only if there exists a proper infinite subgroup $H < G$ such that $H \h G$.

For us, the importance of hyperbolically embedded subgroups stems from the fact that they provide a way to construct quasimorphisms. 
The next statement is an easy consequence of results of Hull--Osin \cite{hull2013induced}.

\begin{lemma}\label{lem:extension of qmorph}
    Suppose that $\{H_\lambda\}_{\lambda \in \Lambda} \h G$ and let $q_\lambda\in Q_h(H_\lambda)$, for $\lambda \in \Lambda$. 
    Then there is a homogeneous quasimorphism $q\in Q_h(G)$ such that the restriction $q|_{H_{\lambda}}=q_\lambda$, for every $\lambda \in \Lambda$.
\end{lemma}
\begin{proof}
    This is a corollary of \cite{hull2013induced}*{Theorem 4.2}. Specifically, quasimorphisms are exactly quasicocycles with values in $\mathbb R$ seen as a trivial $G$-module, and the homogeneous ones are asymmetric quasicocycles.
    Therefore, \cite{hull2013induced}*{Theorem 4.2} together with \cite{hull2013induced}*{Lemma 4.1} yield a quasimorphism $q'\colon G\to \mathbb R$ such that for each $\lambda \in \Lambda$ the restriction of $q'$ to $H_\lambda$ coincides with $q_\lambda$ except at finitely many points (in particular, $q'|_{H_\lambda}$ is close to $q_{\lambda}$). Letting $q \in Q_h(G)$ be the homogeneous quasimorphism close to $q'$ concludes the proof (see \cref{lem:qm_close_to_hqm}).
\end{proof}

Now, suppose that $G$ is an acylindrically hyperbolic group and $X$ is a generating set of $G$ such that $\gax$ is hyperbolic, $G$ acts on $\gax$ acylindrically and $|\partial \gax|>2$.
We denote the set of elements of $G$ acting \emph{loxodromically} on $\gax$ by $\lox(G,X)$.
Note that since the action of $G$ on $\gax$ is acylindrical, every \emph{loxodromic} element $g\in G$ automatically satisfies the \emph{weak proper discontinuity} (WPD) condition, see \cite{osin_acylindrically_2016}. In particular, the set $\lox(G,X)$ coincides with the set $\lox_{WPD}(G,\gax)$, considered in \cite{antolin2016commensurating}.
For the moment, all that we need to know about $\lox(G,X)$ is that if $g^k\in\lox(G,X)$, for \emph{some} $k\in\mathbb Z$, then $g$ has infinite order and $g^n\in\lox(G,X)$, for \emph{every} $n\in\mathbb Z\smallsetminus\{0\}$.

A group is called \emph{elementary} if it is virtually cyclic. Maximal elementary subgroups of loxodromic elements will play an important role in our arguments.

\begin{lemma}[\cite{dahmani2017hyperbolically}*{Lemma 6.5 and Corollary 6.6}] \label{lem:max_elem_sbgp}
   Given $g\in \lox(G,X)$, there exists a unique maximal elementary subgroup of $G$ containing $g$, denoted $E_G(g)$.
   Moreover,
   \[
   E_G(g)=\{x\in G\mid \exists \, m,n \in \mathbb Z\smallsetminus \{0\}\text{ s.t. }xg^nx^{-1}=g^m\}.
   \]
\end{lemma}

For a group $G$, we say that elements $g,h\in G$ are \emph{commensurable} if there is $x \in G$ and $n,m\in \mathbb Z\smallsetminus \{0\}$ such that $g^n = xh^mx^{-1}$. If, additionally, one can ensure that $m=n$ then the elements $g,h$ are called \emph{strongly commensurable}. Both commensurability and strong commensurability define equivalence relations on a group, and an automorphism of $G$ is (strongly) commensurating if and only if it preserves each equivalence class for the corresponding relation setwise.

\begin{lemma}[\cite{antolin2016commensurating}*{Corollary 3.11}]
\label{lem:E(g)_is_hyp_emb} Suppose that $a_1,\dots,a_k \in \lox(G,X)$ is a collection of pairwise non-commensurable loxodromic elements in $G$. Then \[\{E_G(a_1),\dots,E_G(a_k)\} \h (G,X).\]
\end{lemma}

Lemmas~\ref{lem:extension of qmorph} and \ref{lem:E(g)_is_hyp_emb} let us extend quasimorphisms defined on $E_G(g)$ to quasimorphisms of $G$, for every $g\in \lox(G,X)$. However, when $G$ has torsion $E_G(g)$ may fail to have any nontrivial homogeneous quasimorphism at all (this is for instance the case if $E_G(g)$ is isomorphic to the infinite dihedral group). The following provides a class of elements such that $Q_h(E_G(g))$ is obviously large.

\begin{definition}[cf. \cite{antolin2016commensurating}*{Definition 5.11}]\label{def:special}
    An element $g\in G$ is \emph{special} with respect to $X$ if $g\in \lox(G,X)$ and $E_G(g)=\angles g$. We denote the set of all special elements by $S_G(X)$.
\end{definition}

\begin{remark}
    If $g\in S_G(X)$ then $Q_h(E_G(g))=\Hom(\angles g,\mathbb R)\cong\Hom(\mathbb Z,\mathbb R)\cong\mathbb R$.
\end{remark}

Every acylindrically hyperbolic group $G$ contains a unique maximal finite normal subgroup $K(G)$, called the \emph{finite radical of $G$}, see \cite{dahmani2017hyperbolically}*{Theorem~2.24}. The quotient $G/K(G)$ is again acylindrically hyperbolic \cite{minasyan2015acylindrical}*{Lemma~3.9} and has trivial finite radical.
The techniques developed in \cites{antolin2016commensurating,hull2017small} provide a bountiful source of special elements in the case when $K(G)=\{1\}$.

\begin{lem}[\cite{antolin2016commensurating}*{Lemmas~5.13 and 5.5}]\label{lem:product with special is special}
    Suppose that $G$ has no non-trivial finite normal subgroups and $s\in S_G(X)$. For every $x\in G\smallsetminus \angles{s}$ and any finite subset $A \subseteq G$ there is some $M\in\mathbb N$ such that if $\abs{m}>M$ then $s^mx\in S_G(X)$ and $s^m x$ is not commensurable to any element from $A$.
\end{lem}

Given a non-elementary subgroup $H \leqslant G$ such that $H \cap \lox(G,X) \neq\emptyset$, the subgroup 
\[
E_G(H) \coloneqq \bigcap\bigbraces{E_G(h) \bigmid h \in H \cap \lox(G,X)} \leqslant G
\]
is the maximal finite subgroup of $G$ normalized by $H$, see \cite{antolin2016commensurating}*{Lemma~5.6} or \cite{hull2017small}*{Lemma~5.5} (in particular, $E_G(G)=K(G)$). \cref{lem:max_elem_sbgp} implies that $E_G(H)$ contains the centralizer $\cent_G(H)$.

\begin{lem}\label{lem:normal_subgps_are_suitable} If $K(G)=\{1\}$ and $N \lhd G$ is any non-trivial normal subgroup of $G$ then $N$ is non-elementary, $N \cap \lox(G,X) \neq \emptyset$ and $E_G(N)=\cent_G(N)=\{1\}$. Moreover, $S_G(X) \cap N \neq \emptyset$.    
\end{lem}

\begin{proof}
    The fact that $N$ is non-elementary, contains loxodromic elements and satisfies $E_G(N)=\{1\}$ was observed in \cite{chifan2024small}*{Lemma 3.23}.
    The last claim that $S_G(X) \cap N \neq \emptyset$ was proved in \cite{antolin2016commensurating}*{Lemma~5.12} or \cite{hull2017small}*{Corollary 5.7}.  
\end{proof}

\begin{lemma}\label{lem:normal has special}
    Suppose that $K(G)=\{1\}$  and let $\{N_i\}_{i=1}^k$ be a finite collection of non-trivial normal subgroups in $G$.
    Then the intersection $\bigcap_{i=1}^n N_i$ is an infinite normal subgroup of $G$.
    In particular, $S_G(X)\cap \bigcap_{i=1}^n N_i$ is non-empty.
\end{lemma}
\begin{proof}
By the assumptions, each $N_i \lhd G$ is infinite, and since $G$ is acylindrically hyperbolic,  the intersection $N\coloneqq \bigcap_{i=1}^k N_i$ is also an infinite normal subgroup of $G$ (see, for example,  \cite{minasyan2015acylindrical}*{Lemma 6.24}).  The rest of the claim now follows from \cref{lem:normal_subgps_are_suitable}.
\end{proof}

The next theorem is essentially a corollary of the main technical result from \cite{antolin2016commensurating}.

\begin{thm}\label{thm:restr_to_norm_sbgps_is_not_comm}
Suppose that $K(G)=\{1\}$ and $\phi \in \aut(G) \smallsetminus \inn(G)$. Then for every non-trivial normal subgroup $N \lhd G$ there is $h \in \lox(G,X) \cap N$ such that $h$ and $\phi(h)$ are not commensurable in $G$. In particular, $\aut_{sc}(G)=\inn(G)$.
\end{thm}

\begin{proof}
Suppose that for all $h \in \lox(G,X) \cap N$, $h$ is commensurable with $\phi(h)$ in $G$.
 By \cref{lem:normal_subgps_are_suitable},  $N$ is non-elementary, $N \cap \lox(G,X) \neq \emptyset$  and $E_G(N)=C_G(N)=\{1\}$. Therefore, we can apply \cite{antolin2016commensurating}*{Corollary~7.4} to conclude that there exists $w \in G$ such that $\phi(g)=w g w^{-1}$, for all $g \in N$.
 
 After replacing $\phi$ by its composition with conjugation by $w^{-1}$, we can suppose that $w=1$, \emph{i.e.}\ $\phi(g)=g$, for all $g \in N$.
 Now, for every $x \in G$ and arbitrary $g \in N$ we have 
 \[
    xgx^{-1}=\phi(xgx^{-1})=\phi(x) g \phi(x)^{-1},
 \]
 hence $x^{-1}\phi(x) \in \cent_G(N)$. Since $C_G(N)$ is trivial, we conclude that $\phi$ is the identity automorphism, contradicting the assumption that $\phi \notin \inn(G)$.
\end{proof}


\section{Proof of \cref{thm:intro:main_thm}}\label{sec:proof}
As before, throughout this section $G$ will be an acylindrically hyperbolic group with a generating set $X$ such that $\gax$ is hyperbolic, $|\partial \gax|>2$ and the natural action of $G$ on $\gax$ is acylindrical.

We will  need the next two results, that are easy consequences of the Algebraic Dehn Filling Theorem of Dahmani, Guirardel, and Osin \cite{dahmani2017hyperbolically}*{Theorem 7.19}.

\begin{lemma}[\cite{antolin2016commensurating}*{Lemma 8.2}]\label{lem:free wpd}
    Given $g\in\lox(G,X)$, if $\angles{g^n}\triangleleft E_G(g)$ and $\abs{n}$ is large enough, then the normal closure $\aangles{g^n}$ is free and every non-trivial element in it belongs to $\lox(G,X)$.
\end{lemma}

\begin{lemma}\label{lem:dehn filling}
    Suppose that $g\in\lox(G,X)$ is not commensurable to some element $h \in G$. If $\angles{g^n}\triangleleft E_G(g)$ and $\abs{n}$ is large enough, then $\angles{h}\cap\aangles{g^n}=\{1\}$.
\end{lemma}
\begin{proof}
If $h\notin\lox(G,X)$, we can apply \cref{lem:free wpd} to ensure that every non-trivial element of $\aangles{g^n}$ is loxodromic, so $\langle h \rangle \cap \aangles{g^n}=\{1\}$ and we are done.

Thus we can further suppose that $h \in \lox(G,X)$.
Since $h, g\in\lox(G,X)$ are non-commensurable, the pair $\{E_G(h),E_G(g)\}$ is hyperbolically embedded in $(G,X)$, by \cref{lem:E(g)_is_hyp_emb}. According to
    \cite{dahmani2017hyperbolically}*{Theorem 7.19}, there is a finite subset $\mathcal F\subset E_G(g)\smallsetminus \{1\}$ such that if $N\triangleleft E_G(g)$ is disjoint from $\mathcal F$, then $E_G(h)$ intersects trivially $\aangles{N}$, where the normal closure is taken in the whole of $G$.
    Since $E_G(g)$ is virtually cyclic, there is an $M \in \mathbb{N}$ such that $\angles{g^n}\cap \mathcal F=\emptyset$ for every $\abs{n}\geq M$, which concludes the proof.
\end{proof}

The following proposition is the key technical step.

\begin{proposition}\label{prop:special non commensurable}
    Suppose that $K(G)=\{1\}$, $\phi\in \aut(G)\smallsetminus\inn(G)$, $N \lhd G$ is a non-trivial normal subgroup and $A \subseteq G$ is a finite subset. Then there exists $g\in S_G(X) \cap N$ such that $\phi(g)\in S_G(X) \cap N$ and $\angles{\phi(g)}\cap \aangles{g}=\{1\}$. Moreover, 
    $g$ is not commensurable to $\phi(g)$, and neither of these two elements is commensurable to any element of $A$  in $G$.
\end{proposition}
\begin{proof}
    We start by applying \cref{lem:normal has special} and \cref{thm:restr_to_norm_sbgps_is_not_comm} to find \[h\in\lox(G,X) \cap N \cap \phi^{-1}(N)\] such that $h$ is not commensurable to $\phi(h)$.
    By \cref{lem:free wpd} and \cref{lem:dehn filling}, we may choose  $n \in \mathbb{N}$ such that every non-trivial element in $\aangles{h^n}$ is loxodromic and 
  \begin{equation}\label{eq:phi(h)_cap_h^n}
        \angles{\phi(h)}\cap \aangles{h^n}=\{1\}. 
    \end{equation}

    By \cref{lem:normal has special}, there exists a special element \[s\in N \cap \phi(N) \cap \aangles {h^n} \cap \aangles{\phi(h^n)}\cap S_G(X).\]
    Since $s\in \aangles{h^n}$ and $h$ (hence, $\phi(h)$) has infinite order, \eqref{eq:phi(h)_cap_h^n} implies that 
    $\phi(h^n) \notin \angles{s}$, so we can apply \cref{lem:product with special is special} to find some $m \in \mathbb{N}$ such that $s^m\phi(h^n)\in S_G(X)$ and this element is not commensurable to any element from the finite subset $A \cup \phi(A)$. We claim that the $\phi$-preimage of the above element
    \[
        g\coloneqq \phi^{-1}(s^m)h^n \in G
    \]
    has the required properties.

    First, since $h \in N \cap \phi^{-1}(N)$ and $s \in N \cap \phi(N)$, we see that $g$ and $\phi(g)$ are both in $N$.
    In view of \eqref{eq:phi(h)_cap_h^n} and $s\in \aangles{h^n}$, we see that $\angles{\phi(g)}=\angles{s^m\phi(h^n)}$ has trivial intersection with $\aangles{h^n}$.
    Since $s\in \aangles{\phi(h^n)}=\phi(\aangles{h^n})$, it follows that 
   $\phi^{-1}(s) \in \aangles{h^n}$, so
    $g\in \aangles{h^n}$. In particular, $\aangles{g}\leqslant \aangles{h^n}$ intersects trivially $\angles{\phi(g)}$. Since $\phi(g)$ has infinite order, the latter also shows that $g$ is not commensurable to $\phi(g)$ in $G$.

    Since $g$ is a non-trivial element of $\aangles{h^n}$, it is in $\lox(G,X)$, by the choice of $n$. It follows that $g\in S_G(X)$ because $E_G(\phi(g))=\angles{\phi(g)}$ and the property $E_G(g)= \angles g$ is preserved under automorphisms.

    Finally, $\phi(g)=s^m\phi(h^n)$ is not commensurable to an element of $A \cup \phi(A)$, by construction, which also implies that $g$ is not commensurable to any element of $A$ in $G$.
\end{proof}

\begin{thm}\label{thm:main_with_trivial_radical}
    Let $G$ be an acylindrically hyperbolic group with $K(G)=\{1\}$ and $X$ a generating set of $G$ such that $\gax$ is hyperbolic, $|\partial \gax|>2$ and the natural action of $G$ on $\gax$ is acylindrical.
    Then for every $\phi \in \aut(G) \smallsetminus \inn(G)$ there exist $q \in Q_h(G)$ and  $g \in  S_G(X)$ such that
    \begin{enumerate}[(a)]
        \item  $\phi(g) \in S_G(X)$, 
        \item the restriction of $q$ to $\angles{g}$ is unbounded, 
        \item the restriction of $q \circ \phi$ to $\angles{g}$ is zero, 
        \item the difference $q-q\circ \phi:G \to \mathbb{R}$ is not a homomorphism.
    \end{enumerate}
\end{thm}

\begin{proof}
Given $\phi\in \aut(G)\smallsetminus\inn(G)$, we
apply \cref{prop:special non commensurable} to find $g_1\in G$ such that $g_1$ and $\phi(g_1)$ are both special (with respect to $X$) and are not commensurable. Denote $N\coloneqq \aangles{g_1} \lhd G$ and apply \cref{prop:special non commensurable} again to find $g \in S_G(X) \cap N$ such that $\phi(g) \in S_G(X)$ is not commensurable to $g$, and neither of $g$, $\phi(g)$ is commensurable to $g_1$ or $\phi(g_1)$ in $G$. \cref{lem:E(g)_is_hyp_emb} implies that the family $\{\angles{g_1},\angles{\phi(g_1)},\angles{g},\angles{\phi(g)}\}$ is hyperbolically embedded in $G$.

Now, according to \cref{lem:extension of qmorph}, we can
find a homogeneous quasimorphism $q\colon G \to \mathbb{R}$ such that 
\begin{equation}\label{eq:def_of_q}
q(g_1)=q(\phi(g_1))=q(\phi(g))=0~\text{ and } q(g)=1.   
\end{equation}

Clearly, the restriction of $q$ to $\angles{g}$ is unbounded and the restriction of $q\circ\phi$ to $\angles{g}$ is zero. If the difference $r \coloneqq q-q\circ \phi$ was a homomorphism, then $g_1$ would be contained in the kernel of $r$, by \eqref{eq:def_of_q}, so $N=\aangles{g_1} \subseteq \ker(r)$. Since $g \in N$, in view of \eqref{eq:def_of_q} this would imply that
\[
0=r(g)=q(g)-q(\phi(g))=1,
\]
giving a contradiction. Thus $r$ cannot be a  homomorphism, and the proof is complete.
\end{proof}

Evidently, the finite radical $K(G)$ is characteristic in any acylindrically hyperbolic group $G$, so we have a natural homomorphism $\aut(G) \to \aut(G/K(G))$, taking each automorphism of $G$ to the automorphism of $G/K(G)$ induced by it. The next lemma allows us to reduce the proof of \cref{thm:intro:main_thm} from the introduction to \cref{thm:main_with_trivial_radical}.

\begin{lemma}\label{lem:quot_by_rad_is_inner} Let $G$ be an acylindrically hyperbolic group and let $\phi \in \aut(G)$. If $\phi$ induces an inner automorphism of $G/K(G)$ then $\phi$ is strongly commensurating.    
\end{lemma}

\begin{proof} Observe that if $\phi$ induces an inner automorphism of $G/K(G)$ then there exist $w \in G$ and a set map $\varepsilon:G \to K(G)$ such that 
\begin{equation}\label{eq:phi_induces inner}
\phi(g)=wgw^{-1} \varepsilon(g),~\text{ for all } g \in G.    
\end{equation}
Since $K(G) \lhd G$ is a finite normal subgroup, its centralizer $\cent_G(K(G))$ is a finite index normal subgroup of $G$. Denote $k=|K(G)|$ and $l=|G:\cent_G(K(G))|$, and consider any element $g \in G$. Then $wg^lw^{-1} \in \cent_G(K(G))$, so  \eqref{eq:phi_induces inner} implies that 
\[\phi(g^{lk})=(\phi(g^l))^k= \left(wg^lw^{-1}\varepsilon(g^l)\right)^k=\left(wg^lw^{-1}\right)^k(\varepsilon(g^l))^k=w g^{lk} w^{-1}.\]
Thus $\phi$ is strongly commensurating.    
\end{proof}

\begin{proof}[Proof of \cref{thm:intro:main_thm}]
Evidently, (iii) implies (ii), and (ii) implies (i) by \cref{lem:sc_auts_form_a_normal_sbgps}.
To show that (i) implies (iii), assume that $\phi \in \aut(G) \smallsetminus \aut_{sc}(G)$. Let $\overline{G}$ denote the quotient $G/K(G)$ and let $\psi:G \to \overline{G}$ be the quotient map. 

Then $\overline{G}$ is acylindrically hyperbolic and has trivial finite radical. By \cref{lem:quot_by_rad_is_inner},  the automorphism $\overline{\phi} \in \aut(\overline{G})$, induced by $\phi$,  is not inner. Therefore, we can apply \cref{thm:main_with_trivial_radical} to find an infinite order element $\overline{g} \in \overline{G}$ and a homogeneous quasimorphism $\overline{q} \in Q_h(\overline{G})$ such that $\overline{q}|_{\angles{\overline{g}}}$ is unbounded, $(\overline{q}\circ \overline{\phi})|_{\angles{\overline{g}}}$ is zero, and $(\overline{q} - \overline{q}\circ \overline{\phi})\colon \overline{G}\to\mathbb R$ is not a homomorphism.

We can now define $q:G \to \mathbb{R}$ by $q=\overline{q} \circ \psi$ and let $g \in G$ be any preimage of $\overline{g}$. It is easy to check that $q \in Q_h(G)$ satisfies (iii), so the proof is complete.
\end{proof}


\section{Coarse automorphisms}\label{sec:crse aut}
We shall now explain how \cref{thm:intro:main_thm} implies that $\out(G)$ embeds into groups of coarse automorphisms of $G$. To do this, we need to explain some coarse geometric preliminaries.

Coarse geometry can be described as the study of geometric properties that are invariant under uniformly bounded perturbations. A coarse structure on a set $X$ is a gadget used to describe what ``uniformly bounded'' means on $X$, much like the way a ``uniform structure'' is used to describe uniformly small sets. We now briefly recall some notions and notation of coarse geometry and group theory, using notation from \cite{coarse_groups}. We refer to that work for more details and motivation.

\begin{definition}
    A \emph{coarse structure} $\mathcal E$ on a set $X$ is a collection of subsets of $X\times X$ that is closed under taking subsets and finite unions, and such that:
    \begin{enumerate}
        \item $\Delta_X \coloneqq \{(x,x)\mid x\in X\}$ belongs to $\mathcal E$;
        \item if $E\in\mathcal E$, then $\op{E}\coloneqq\{(y,x)\mid (x,y)\in E\}$ belongs to $\mathcal E$;
        \item if $E,F\in\mathcal E$, then 
        $
            E\cmp F\coloneqq\{(x,z)\mid \exists y\in X\text{ s.t. }(x,y)\in E,\ (y,z)\in F\}
        $
        belongs to $\mathcal E$.
    \end{enumerate}
    A \emph{coarse space} is a set with a coarse structure.
\end{definition}

Following \cite{coarse_groups}, we may use bold symbols to denote coarse notions, so that a coarse space $(X,\mathcal E)$ is denoted by $\crse X$.
A function $f\colon(X,\mathcal E)\to (Y,\mathcal F)$ is \emph{controlled} if $(f\times f)(E)\in \mathcal F$, for every $E\in\mathcal E$. Two functions $f,f'\colon X\to (Y,\mathcal F)$ are \emph{close} if $(f\times f')(\Delta_X)\in\mathcal F$.
A \emph{coarse map} $\crse{f\colon X\to Y}$ is the equivalence class of a controlled function $f\colon (X,\mathcal E)\to (Y,\mathcal F)$, where functions are equivalent when they are close. Two coarse spaces $\crse{X,Y}$ are \emph{coarsely equivalent} if there are coarse maps $\crse{f\colon X\to Y}$ and $\crse{g\colon Y\to X}$ such that $\crse{g\cmp f=\cid_{X}}$ and $\crse{f\cmp g=\cid_{Y}}$. It is an exercise to show that coarse spaces and coarse maps define a category, and coarse equivalences are the isomorphisms in this category.

\begin{example}\label{exmp:metric coarse}
    If $d$ is a metric on a set $Z$ then
    \[
        \mathcal E_d\coloneqq\{E\subseteq Z\times Z \mid \exists\, r\geq 0,\ (x,y)\in E \implies d(x,y)\leq r\}
    \]
    is a coarse structure on $Z$. If $(X,d_X)$ and $(Y,d_Y)$ are metric spaces, closeness of functions and coarse equivalence between $(X,\mathcal E_{d_X})$ and $(Y,\mathcal E_{d_Y})$ coincide with the usual metric notions (see \emph{e.g.}\ \cite{nowak2012large}*{Section 1.4}).
\end{example}

Let now $G$ be a group. Given $E\subseteq G\times G$, denote its translates by 
\[
    g\cdot E\coloneqq \{(gx,gy)\mid (x,y)\in E\}
    \quad \text{ and }\quad
    E\cdot g\coloneqq \{(xg,yg)\mid (x,y)\in E\}. 
\]

\begin{definition}
    A coarse structure $\mathcal E$ on $G$ is \emph{equi bi-invariant} if
    \[
    E\in\mathcal E\implies
    \left(\bigcup_{g,h\in G} g\cdot E\cdot h\right) \in \mathcal E.
    \]
\end{definition}

The importance of equi bi-invariance for $\crse G = (G,\mathcal E)$ is that it is equivalent to the requirement that the multiplication function is controlled, and hence defines a coarse map $\crse{G\times G\to G}$. When this is the case, $\crse G$ is a group object in the category of coarse spaces.

\begin{remark}
    It is not hard to show that the coarse structure $\mathcal E_d$ induced by a metric $d$ on $G$ (\cref{exmp:metric coarse}) is equi bi-invariant if and only if $d$ is coarsely equivalent to a bi-invariant metric (\cite{coarse_groups}*{Lemma 8.2.1}). One can also observe that if $d$ is bi-invariant then $(G,d)$ is a group object in the category of metric spaces and Lipschitz maps, so $\crse G = (G,\mathcal E_d)$ is the ``coarsification'' of $(G,d)$.
\end{remark}

\begin{remark}
    The group objects in the category of coarse spaces are called \emph{coarse groups}. In principle, coarse groups need not be groups, as the multiplication function only needs to satisfy the  group axioms up to closeness. This means that, for instance, the associative law only needs to hold up to uniformly boundend error: $(g_1g_2)g_3\approx g_1(g_2g_3)$.
    Examples of coarse groups that are not a group can be constructed by considering approximate subgroups of a group with a bi-invariant metric \cite{coarse_groups}*{Corollary 5.3.6}.
    However, it is still an open problem to exhibit a coarse group that is not \emph{coarsely isomorphic} to a coarsified group \cite{coarse_groups}*{Question 11.0.1}.
\end{remark}

A \emph{coarse homomorphism} between groups with equi bi-invariant coarse structures $\crse G= (G,\mathcal E)$, $\crse H=(H,\mathcal F)$ is a coarse map $\crse{\phi\colon G\to H}$ such that
\[
    \{(\phi(g_1g_2),\phi(g_1)\phi(g_2))\mid g_1,g_2\in G\} \in\mathcal F,
\]
where $\phi\colon G\to H$ is any representative of $\crse \phi$.

\begin{example}
    If $G$ is a finitely normally generated group, let $\varcrs{bw}$ be the coarse structure defined by the bi-invariant word metric $\dbw$ associated with the closure under conjugacy of a finite normally generating set $S$ (as explained in the introduction). It is worthwhile to note that, while $\dbw$ depends on $S$ and is, hence, only defined up to a Lipschitz equivalence, $\varcrs{bw}$ does not depend at all on the choice of the finite normally generating set.
    In this setup, a map $q\colon (G,\varcrs{bw})\to (\mathbb R,\varcrs{\abs{\variable}})$ defines a coarse homomorphism if and only if it is a quasimorphism.
    
    Here the assumption of finite normal generation is used to define $\dbw$, which is a `nice' bi-invariant metric for which every quasimorphism $G\to \mathbb R$ is Lipschitz and so it is controlled as a map between metric spaces $(G,\dbw)\to (\mathbb R,\abs{\variable})$.
    For a general group $G$ such a metric will not exist.
    
    In general, for an arbitrary group with an equi bi-invariant coarse structure $(G,\mathcal E)$, any representative $q$ of a coarse homomorphism $\crse q\colon (G,\mathcal E)\to(\mathbb R,\varcrs{\abs\variable})$ must be a quasimorphism, but the converse is not true because a $q\in Q(G)$ may fail to be controlled.    
    One can still define a minimal `connected' equi bi-invariant coarse structure $\varcrs[grp]{fin}$ on $G$ by closing the set of finite subsets of $G\times G$ under bilateral translations (\cite{coarse_groups}*{Section 4.4}), and it is once again the case that $q\colon (G,\varcrs[grp]{fin})\to (\mathbb R,\varcrs{\abs{\variable}})$ defines a coarse homomorphisms if and only if $q$ is a quasimorphism. If $G$ is finitely normally generated, then $\varcrs[grp]{fin}=\varcrs{bw}$.
\end{example}

Naturally, a \emph{coarse isomorphism} is a coarse homomorphisms $\crse{\phi\colon G\to H}$ for which there exists an inverse coarse homomorphism $\crse{\phi^{-1}\colon H\to G}$. Equivalently, it is easy to show that a coarse isomorphism is a coarse homomorphism that is also a coarse equivalence (\cite{coarse_groups}*{Section 5.2}).

\begin{definition}
    For $\crse{G}=(G,\mathcal E)$ group with equi bi-invariant coarse structure, we denote by $\cAut(\crse G) = \cAut(G,\mathcal E)$ the group of \emph{coarse automorphisms} of $\crse G$.
\end{definition}

We emphasize that $\cAut(\crse{G})$ consists of \emph{equivalence classes} of maps.

\begin{remark}
    If $d$ is a quasi-geodesic bi-invariant metric on $G$ then $\cAut(G,\mathcal E_d)$ is a subgroup of $\qi(G,d)$, because every coarse automorphisms is a coarse equivalence and coarse equivalences between quasi-geodesic metric spaces are quasi-isometries (see \emph{e.g.}\ \cite{coarse_groups}*{Appendix B.1}).
    In particular, when $G$ is finitely normally generated this applies to $\dbw$.
\end{remark}

We are now ready to come back to the main topic of this note. Let $G$ be a group and define
\[
\mathcal E_{G\to\mathbb R}\coloneqq
\left\{E\subseteq G\times G\ \middle|\ \forall\, q\colon G\to \mathbb R\text{ quasimorphism, }
(q\times q)(E)\in \varcrs{\abs\variable}
\right\}
\]
(this is the largest coarse structure for which all quasimorphisms are controlled).

Recall from the introduction that an HS-quasimorphism is a map $f\colon G\to H$ that pulls back quasimorphisms of $H$ to quasimorphisms of $G$; two HS-quasimorphisms are equivalent if they cannot be distinguished up to closeness by post-composing with quasimorphisms; and that $\qout(G)$ is the group of equivalence classes of invertible HS-quasimorphisms of $G$ with itself.
It is then not hard to prove the following.

\begin{proposition}[\cite{coarse_groups}*{Proposition 12.3.2}] 
    Let $G$ and $H$ be groups. The coarse structure 
    $\mathcal E_{G\to\mathbb R}$ is equi bi-invariant. Moreover, a function $f\colon G\to H$ is an HS-quasimorphism if and only if it defines a coarse homomorphism $(G, \mathcal E_{G\to\mathbb R})\to(H,\mathcal E_{H\to\mathbb R})$, and two such HS-quasimorphisms are close if and only if they are equivalent.
    In particular, $\qout(G)=\cAut(G,\mathcal E_{G\to\mathbb R})$.
\end{proposition}

Let us now recall the following result, which is essentially proved in \cite{coarse_groups}*{Proposition 5.2.14}.

\begin{prop}\label{prop:cAut to QOUT}
    Let $\mathcal E$ be an equi bi-invariant coarse structure on $G$. The following are equivalent:
    \begin{enumerate}
        \item $\mathcal E\subseteq \mathcal E_{G\to\mathbb R}$;
        \item fixing representatives defines a natural homomorphism
        \[
        \cAut(G,\mathcal E)\to 
        \cAut(G,\mathcal E_{G\to\mathbb R}) =\qout(G).
        \]
    \end{enumerate}
\end{prop}
\begin{proof}
    For $(1)\Rightarrow (2)$, let $\crse{\phi\colon (G,\mathcal E)\to (G,\mathcal E)}$ be a coarse homomorphisms and fix a representative $\phi$ for it. We need to show that $\phi$ is an HS-quasimorphism and that other representatives for $\crse\phi$ define equivalent HS-quasimorphisms.
    For the first part, fix any quasimorphism $q\in Q(G)$. By definition, for every $F\in \mathcal E_{G\to\mathbb R}$ the image $q\times q(F)$ is in $\varcrs{\abs\variable}$, that is, $q\colon (G, \mathcal E_{G\to\mathbb R})\to (\mathbb R,\varcrs{\abs\variable})$ is controlled and hence defines a coarse homomorphism.    
    Since $\mathcal E\subseteq \mathcal E_{G\to\mathbb R}$, the maps
    \[
    (G,\mathcal E)\xrightarrow{\ \phi \ }
    (G,\mathcal E) \xrightarrow{\ \id \ }
    (G,\mathcal E_{G\to\mathbb R}) \xrightarrow{\ q \ }
    (\mathbb R,\varcrs{\abs\variable})
    \]
    represent a composition of coarse homomorphisms, which implies that $q\circ \phi$ is a (controlled) quasimorphism. This proves that $\phi$ is a HS-quasimorphism. The second part is also immediate, because composition of coarse homomorphisms is well defined: a different choice of a representative $\phi'$ for $\crse \phi$ will result in the same coarse homomorphisms $\crse{q\cmp\phi}$, and, unravelling the definitions, this means that  $\phi$ and $\phi'$ cannot be told apart using quasimorphisms, \emph{i.e.}\ they are equivalent.
    
    For $(2)\Rightarrow (1)$, if $\mathcal E\nsubseteq \mathcal E_{G\to\mathbb R}$, it follows that there exists some subset $B\subseteq G$ and $q\in Q(G)$ such that $q(B)$ is unbounded in $\mathbb R$ but $B\times B\in \mathcal E$ (\cite{coarse_groups}*{Proposition~4.5.2}, the ``in particular'' statement). Arbitrarily choose a function $\phi\colon B\to B$ such that $\sup_{x\in B}\abs{q\circ \phi(x) - q(x)}=\infty$, and extend it as the identity on $G\smallsetminus B$. 
    Then $\phi$ is close to $\id_G$ with respect to $\mathcal E$, so $\crse{\phi=\cid_{G}}\in\cAut(G,\mathcal E)$, but $\phi$ is not an HS-quasimorphism.
\end{proof}

Now we need two more definitions: a coarse structure $\mathcal E$ on $G$ is \emph{$\aut(G)$-invariant} if $(\phi\times\phi)(E)\in \mathcal E$, for all $\phi\in \aut(G)$ and $E\in\mathcal E$.\footnote{%
    We are not requiring that $\bigcup_{\phi\in\aut(G)}(\phi\times\phi)(E)\in\mathcal E$: this is precisely the difference between `invariance' and `equi invariance'.}
If $\mathcal E$ is also equi bi-invariant, this means precisely that sending $\phi$ to $\crse \phi$ defines a natural homomorphism $\Psi_{\mathcal E}\colon\aut(G)\to\cAut(G,\mathcal E)$.
And finally, a coarse structure $\mathcal E$ on $G$ is \emph{connected} if it contains every finite subset of $G \times G$.

\begin{proof}[Proof of \cref{cor:intro:cAut}]
    Since $\mathcal E$ is $\aut(G)$-invariant and equi bi-invariant, the homomorphisms $\Psi_{\mathcal E}\colon \aut(G)\to\cAut(G,\mathcal E)$ is well-defined. Since $\mathcal E$ is also connected, conjugations are close to the identity. In fact, $\{(g,1)\},\{(g^{-1},1)\}\in\mathcal E$ together with equi bi-invariance imply that
    \[
    \{(gx,x)\mid x\in G\}\in\mathcal E, \qquad
    \{(gxg^{-1},gx)\mid x\in G\}\in\mathcal E,
    \]
    and the claimed closeness follows from the fact that $\mathcal E$ is closed under composition
    (this argument is simply translating \eqref{eq:conjugation is close} to the language of coarse structures).
    It follows that $\Psi_{\mathcal E}$ factors through $\out(G)$, mapping the outer automorphism $[\phi]$ to the coarse automorphism $\crse \phi$.

    Since $\mathcal E\subseteq \mathcal E_{G\to \mathbb R}$, \cref{prop:cAut to QOUT} shows that sending $\crse \phi$ to the equivalence class of $\phi$ is a well defined homomorphisms $\cAut(G,\mathcal E)\to\qout(G)$.
    The corollary follows, because \cref{cor:intro:qout} shows that for  acylindrically hyperbolically groups with no non-trivial finite normal subgroups the composition
    \[
    \out(G)\to\cAut(G,\mathcal E)\to \qout(G)
    \]
    is injective.
\end{proof}

\begin{remark}
    Since $\cAut(G,\mathcal E_d)\leqslant \qi(G,d)$ when $d$ is a bi-invariant metric, and $\mathcal E_{\dbw}\subseteq \mathcal E_{G\to\mathbb R}$ when $G$ is also finitely normally generated, \cref{cor:intro:qi} directly follows from \cref{cor:intro:cAut}.
\end{remark}


\section{Concluding thoughts}
We conclude by remarking that it is unclear how much \cref{cor:intro:cAut} is more general than \cref{cor:intro:qout}. It is in principle possible that there is a large number of pairwise different connected $\aut(G)$-invariant equi bi-invariant coarse structures on $G$, all contained in $\mathcal E_{G\to\mathbb R}$. In that case, $\cAut(G,\mathcal E)$ could be a whole suite of different groups. However we do not know if this is the case.

The situation is already mysterious enough for the  free group $F_n$ of rank $n \ge 2$. Here one can see that $\varcrs{bw}=\varcrs[grp]{fin}$ is one such coarse structure, and it is not hard to show that it is different from $\mathcal E_{G\to\mathbb R}$, because $(F_n,\varcrs{bw})$ is not coarsely abelian (\cite{coarse_groups}*{Corollary 10.1.5}). It is however unclear what other coarse structures might there be. This raises the following:

\begin{problem}\label{prob:classify coarse structures}
    Classify the $\aut(F_n)$-invariant equi bi-invariant coarse structures $\mathcal E$ on $F_n$ with $\varcrs{bw}\subseteq \mathcal E\subseteq \mathcal E_{F_n\to\mathbb R}$.
\end{problem}

\begin{remark}
    Note that if $d$ is an $\aut(F_n)$-invariant left-invariant metric on $F_n$, then $\mathcal E_d$ is $\aut(F_n)$-invariant equi bi-invariant. Moreover, $\mathcal E_d$ always contains $\varcrs{bw}$. However, it will be generally difficult to understand whether $\mathcal E_d$ is contained in $\mathcal E_{F_n\to\mathbb R}$. This brings back to the questions about the existence of $\aut$-invariant quasimorphisms mentioned in the introduction (see \cite{fournier2023Aut}).
\end{remark}

Note that a solution to \cref{prob:classify coarse structures} would not immediately clarify what the induced groups $\cAut(F_n,\mathcal E)$ are. In general, there is no reason to expect that there be homomorphisms $\cAut(F_n,\mathcal E)\to\cAut(F_n,\mathcal E')$, not even if $\mathcal E\subseteq \mathcal E'$ (a map that is controlled from $(G,\mathcal E)$ to $(G,\mathcal E)$ may fail to be controlled  from $(G,\mathcal E')$ to $(G,\mathcal E')$). We thus raise this as an independent problem:

\begin{problem}
    Classify the groups  $\cAut(F_n,\mathcal E)$, where $\mathcal E$ is a $\aut(F_n)$-invariant equi bi-invariant coarse structure with $\varcrs{bw}\subseteq \mathcal E\subseteq \mathcal E_{F_n\to\mathbb R}$.
\end{problem}
\begin{remark}
    As a matter of fact, we do not even know whether the homomorphism
    $
    \cAut(F_n,\varcrs{bw})\to \qout(F_n)
    $
    is an isomorphism.
\end{remark}

\medskip

We conclude this note with one first step in the direction of \cref{prob:classify coarse structures}.
Let $G$ be a group normally generated by a finite set $S$. We can define the \emph{coarse abelianization} of $(G,\varcrs{bw})$ as the minimal equi bi-invariant coarse structure $\varcrs{ab}$ containing $\varcrs{bw}$ and $\{\text{commutators}\}\times\{\text{commutators}\}$.
In practice, this is the coarse structure induced by the word metric with respect to the generating set
\begin{equation}\label{eq:d_ab}
 \bigcup_{g\in G}S^g \cup \{[g,h]\mid g,h \in G\}.   
\end{equation}

We denote this metric by $d_{\rm ab}$, so that $\varcrs{ab}=\mathcal E_{d_{\rm ab}}$.

It is clear that $\varcrs{ab}$ is equi bi-invariant and connected.
It is $\aut(G)$-invariant, because the image of $S$ under an automorphism $\phi$ is contained in $S^n$ for some $n\in\mathbb N$, hence $\phi(S^g)\subseteq (S^{\phi(g)})^n$ for every $g\in G$, and the set of commutators is $\aut(G)$-invariant.
It is also clear that $\varcrs{bw}\subseteq\varcrs{ab}\subseteq\mathcal E_{G\to\mathbb R}$. For the free group, we also know that the first inclusion is strict, because $(F_n,\varcrs{bw})$ is not coarsely abelian \cite{coarse_groups}*{Corollary 10.1.5}.
As it turns out, proving that the second containment is also strict is a much subtler question.

\smallskip

We first need a lemma, which was kindly explained to us by Jarek K\k{e}dra. On the derived subgroup $[F_n,F_n]\trianglelefteq F_n$, consider both the restriction of the metric $\dbw$ associated with the free generating set $S$  of $F_n$, and the \emph{commutator length}---that is, the word metric with respect to the set of commutators $\{[g,h]\mid g,h \in F_n\}$. Let $\abs\variable_{\rm bw}$ and $\abs\variable_{\rm cl}$ denote the associated length functions. 

\begin{lem}\label{lem:commutator vs cancellation length}
    For every $g\in [F_n,F_n]$, we have $2\abs{g}_{\rm cl}\leq \abs{g}_{\rm bw}$.
\end{lem}
\begin{proof}
    Let $S$ be a free generating set of $F_n$ and suppose that $\abs{g}_{\rm bw} =k$. Then
    \begin{equation}\label{eq:g as prod of conj}
    g = s_1^{h_1}\cdots s_k^{h_k}
    \end{equation}
    with $s_i\in S^{\pm 1}$ and $h_i \in F_n$. Since $g\in [F_n,F_n]$, its image vanishes in the abelianization $\mathbb Z^n$. This implies that there must be some $1<j\leq k $ such that $s_j = s_1^{-1}$ (and it also implies that $k$ must be even). We rewrite $g = s_1^{h_1}x (s_1^{-1})^{h_j}y$, where $x,y$ represent the products of the remaining conjugates from \eqref{eq:g as prod of conj}. We then have
    \begin{align*}
        g &= s_1^{h_1} x h_j s_1^{-1} h_j^{-1} y \\
        & = s_1^{h_1} x h_j 
         (h_1^{-1} h_1) s_1^{-1} (h_1^{-1} h_1) h_j^{-1} 
         (x^{-1} x)  y \\
        & = [s_1^{h_1}, x h_jh_1^{-1}] xy.
    \end{align*}
    By definition, $xy$ is a product of $k-2$ conjugates of the generators. By induction, $xy$ satisfies $k-2 \geq \abs{xy}_{\rm bw}\geq 2\abs{xy}_{\rm cl}$, and hence $k \geq 2(\abs{xy}_{\rm cl} +1)\geq 2\abs{g}_{\rm cl}$.
\end{proof}

\cref{lem:commutator vs cancellation length} lets us easily compare the commutator length with the length $\abs{\variable}_{\rm ab}$ defined by $d_{\rm ab}$.

\begin{corollary}\label{cor:abelian vs commutator length}
    For every $g\in [F_n,F_n]$, $\abs{g}_{\rm ab}=\abs{g}_{\rm cl}$, where $\abs\variable_{\rm ab }$ is the word metric corresponding to the generating set \eqref{eq:d_ab}, for a free generating set $S$ of $G=F_n$.
\end{corollary}
\begin{proof}
    A path from $1$ to $g$ is a product of conjugates of the generators of $F_n$ and commutators. Up to conjugating the commutators, we may assume that they all come last:
    \[
    g = s_1^{h_1}\cdots s_k^{h_k}[g_1,g_1']\cdots[g_l,g_l'],
    \]
    where $s_i\in S^{\pm 1}$, $h_i,g_i,g_i'\in F_n$, for $i=1,\dots,k$.
    If $k\neq 0$, \cref{lem:commutator vs cancellation length} implies that we can obtain a strictly shorter path replacing $s_1^{h_1}\cdots s_k^{h_k}$ by an appropriate product of commutators. In particular, a shortest path in the metric $d_{\rm ab}$ must be geodesic also for $d_{\rm cl}$.
\end{proof}

We are now in position to apply a highly non-trivial theorem of Kharlampovich--Myasnikov \cite{kharlampovich2001implicit} to prove the following.

\begin{thm}\label{thm:strict containment}
    For the free group $F_n$ with $n\geq 2$ the inclusion $\varcrs{ab}\subset \mathcal E_{F_n\to\mathbb R}$ is strict.
\end{thm}
\begin{proof}
    By \cite{kharlampovich2001implicit}*{Theorem 3}, there exists a (non-explicit!) set $B\subset [F_n,F_n]$ consisting of elements of arbitrarily large commutator length, such that the squares $g^2$, with $g\in B$, have uniformly bounded commutator length (see also \cite{bartholdi2024commutator}*{Section 1.2}).
    In particular, for every homogeneous quasimorphism $q\in Q_h(F_n)$ and all $g \in B$  we have $q(g)=\frac{q(g^2)}{2}=0$.
    This shows that $B$ is bounded with respect to $\mathcal E_{F_n\to\mathbb R}$ (\emph{i.e.}\ $B\times B\in \mathcal E_{F_n\to\mathbb R}$). 
    
    On the other hand, \cref{cor:abelian vs commutator length} shows that $B$ has unbounded diameter with respect to $d_{\rm ab}$, and it is hence not bounded with respect to $\varcrs{ab}$.
\end{proof}

This answers a question implicit in \cite{coarse_groups}*{Remark 12.3.13} in the special case of free groups.

\begin{remark}
    It is shown in \cites{coulon2025first,andre2022formal} (see \cite{coulon2025first}*{Remark 7.27}) that every acylindrically hyperbolic group contains elements of arbitrarily large commutator length whose squares have uniformly bounded commutator length.
\end{remark}

\bibliography{biblio.bib}

\end{document}